\documentclass[12pt]{article}
\usepackage{amsmath,amssymb,amsthm,comment}

\newtheorem{theorem}{Theorem}[section]
\newtheorem{thmy}{Theorem}
\newtheorem{lemma}[theorem]{Lemma}
\newtheorem{corollary}[theorem]{Corollary}

\def\barr{\begin{array}}
\def\earr{\end{array}}

\title{Finite groups whose subgroup graph contains a vertex of large degree}
\author{Marius T\u arn\u auceanu}
\date{February 4, 2025}

\begin{document}

\maketitle

\begin{abstract}
T.C. Burness and S.D. Scott \cite{3} classified finite groups $G$ such that the number of prime order subgroups of $G$ is greater than $|G|/2-1$. In this note, we study finite groups $G$ whose subgroup graph contains a vertex of degree greater than $|G|/2-1$. The classification given for finite solvable groups extends the work of Burness and Scott.
\end{abstract}

{\small
\noindent
{\bf MSC2000\,:} Primary 20D60; Secondary 20D15, 05C07.

\noindent
{\bf Key words\,:} finite group, subgroup graph, vertex degree.}

\section{Introduction}

Recently, mathematicians constructed many graphs which are assigned to groups by different methods (see e.g. \cite{4}). One of them is the \textit{subgroup graph} $L(G)^*$ of a group $G$, that is the graph whose vertices are the subgroups of $G$ and two vertices, $H_1$ and $H_2$, are connected by an edge if and only if $H_1\leq H_2$ and there is no subgroup $K\leq G$ such that $H_1<K<H_2$. Finite groups with planar subgroup graph have been determined by Starr and Turner \cite{12}, Bohanon and Reid \cite{1}, and Schmidt \cite{10}. Also, the genus of $L(G)^*$ has been studied by Lucchini \cite{8}. A remarkable subgraph of $L(G)^*$ is the \textit{cyclic subgroup graph} of $G$, which has been investigated in \cite{13}. Note that the (cyclic) subgroup graph of a group $G$ is in fact the Hasse diagram of the poset of (cyclic) subgroups of $G$, viewed as a simple, undirected graph.

The starting point for our discussion is given by the paper \cite{3}, which classifies the nontrivial finite groups $G$ with the property $\delta(G)>|G|/2-1$, where $\delta(G)$ is the number of prime order subgroups of $G$. Recall that a \textit{generalized dihedral group} is a group of the form $D(A)=A\langle\tau\rangle=A.2$, where $A$ is abelian and $\tau$ acts by inversion. We mention that $E$, $C_n$, $D_8$ denote an elementary abelian $2$-group, a cyclic group of order $n$ and the dihedral group of order $8$, respectively. 

\begin{thmy}
For a nontrivial finite group $G$, we have $\delta(G)>|G|/2-1$ if and only if one of the following holds:
\begin{itemize}
\item[{\rm (I)}] $G=D(A)$, where $A$ is abelian;
\item[{\rm (II)}] $G=D_8\times D_8\times E$, where $E$ is an elementary abelian $2$-group;
\item[{\rm (III)}] $G=H(r)\times E$, where $H(r)\cong(D_8\times\cdots\times D_8)/C_2^{r-1}$ is a central product of $r\geq 1$ copies of $D_8$ so that
\begin{align*}
H(r)=&\langle x_1,y_1,...,x_r,y_r,z\mid x_i^2=y_i^2=z^2=1, \mbox{ all pairs of generators }\\
&\mbox{ commute except } [x_i,y_i]=z\rangle;
\end{align*}
\item[{\rm (IV)}]  $G=S(r)\times E$, where $S(r)$ is the split extension of an elementary abelian group of order $2^{2r}$ ($r\geq 1$) by a cyclic group $C_2=\langle z\rangle$ so that
\begin{align*}
S(r)=&\langle x_1,y_1,...,x_r,y_r,z\mid x_i^2=y_i^2=z^2=1, \mbox{ all pairs of generators }\\
&\mbox{ commute except } [z,x_i]=x_iy_i\rangle;
\end{align*}
\item[{\rm (V)}] $G=T(r)$ is the split extension of an elementary abelian group $A$ of order $2^{2r}$ ($r\geq 1$) by a cyclic group $C_3=\langle z\rangle$ so that
\begin{align*}
T(r)=&\langle x_1,y_1,...,x_r,y_r,z\mid x_i^2=y_i^2=z^3=1, \mbox{ all pairs of generators }\\ 
&\mbox{ commute except } [z,x_i]=x_iy_i \mbox{ and } [z,y_i]=x_i\rangle;
\end{align*}
\item[{\rm (VI)}] $G$ is a group of exponent $3$;
\item[{\rm (VII)}] $G=S_3\times D_8\times E$;
\item[{\rm (VIII)}] $G=S_3\times S_3$;
\item[{\rm (IX)}] $G=S_4$;
\item[{\rm (X)}] $G=A_5$.
\end{itemize}
\end{thmy}
\smallskip

In \cite{14}, C.T.C. Wall classifies the finite groups $G$ with the property $i_2(G)>|G|/2-1$, where $i_2(G)$ is the number of involutions in $G$. Since $\delta(G)\geq i_2(G)$, Theorem A generalizes Wall’s result. More precisely, the groups obtained by Wall are those of types (I)-(IV) in Theorem A.

We remark that $\delta(G)$ is the degree of the trivial subgroup in the subgroup graph of $G$. Thus a natural extension of Theorem A is to classify the finite groups $G$ such that $d_{L(G)^*}(H)>|G|/2-1$ for some subgroup $H$ of $G$. 

Our main result is as follows.

\begin{theorem}
Let $G$ be a nontrivial finite solvable group whose subgroup graph contains a vertex of degree strictly greater than $|G|/2-1$. Then $G$ belongs to at least one of the following families of groups:
\begin{itemize}
\item[{\rm 1)}] The groups of order at most $11$ except $C_n$ for $n=5,...,11$.
\item[{\rm 2)}] The groups (I)-(IX) in Theorem A.
\item[{\rm 3)}] Elementary abelian $2$-groups.
\item[{\rm 4)}] $C_2^{s-1}\times C_4$, where $s\geq 1$.
\item[{\rm 5)}] Generalized extraspecial $2$-groups.
\item[{\rm 6)}] $C_p^n\rtimes C_2$, where $p$ is an odd prime and $n$ is a positive integer.
\item[{\rm 7)}] $D_{12}$.
\end{itemize}
\end{theorem}

Note that, among these groups, there are two families not contained in Theorem A: $C_2^{s-1}\times C_4$ with $s\geq 1$ and generalized extraspecial $2$-groups which are not of type (III) in reference to the classification of Theorem A. Recall that a finite $2$-group is called \textit{generalized extraspecial} if $G'=\Phi(G)$ has order $2$ and $G'\subseteq Z(G)$. Here $G'$, $\Phi(G)$, $Z(G)$ denote the derived subgroup, the Frattini subgroup and the centre of $G$, respectively. The structure of these groups is well-known (see e.g. Theorem 2.3 of \cite{2} and Lemma 3.2 of \cite{11}).

We will  prove that for a solvable group $G$, if a vertex $H$ of the subgroup graph $L(G)^*$ has "large degree" (i.e. $d_{L(G)^*}(H)>|G|/2-1$), then this vertex lies at the very top or at the very bottom of the Hasse diagram of the subgroup lattice of $G$. For the proof of Theorem 1.1, we need the following theorem which collects some results on the number of maximal subgroups of a finite solvable group.

\begin{thmy}
Let $G$ be a finite solvable group, ${\rm Max}(G)$ be the set of maximal subgroups of $G$ and $\pi(G)$ be the set of primes dividing the order of $G$. For every $p\in\pi(G)$, we denote by ${\rm Max}_p(G)$ the set of maximal subgroups of $G$ whose index is a power of $p$ and by ${\bf O}^p(G)$ the smallest normal subgroup of $G$ whose index is a power of $p$. Then the following statements hold:
\begin{itemize}
\item[{\rm a)}]{\rm (G.E. Wall \cite{15})} $|{\rm Max}(G)|\leq |G|-1$.
\item[{\rm b)}]{\rm (R.J. Cook, J. Wiegold and A.G. Williamson \cite{5})} If $p=\min\pi(G)$, then
\begin{equation}
|{\rm Max}(G)|\leq\frac{|G|-1}{p-1}\,.\nonumber
\end{equation}Moreover, we have equality if and only if $G$ is elementary abelian.
\item[{\rm c)}]{\rm (M. Herzog and O. Manz \cite{6})} If $p=\min\pi(G)$ and $q=\max\pi(G)$, then
\begin{equation}
|{\rm Max}(G)|\leq\frac{q|G/\Phi(G)|-p}{p(q-1)}\,.\nonumber
\end{equation}
\item[{\rm d)}]{\rm (B. Newton \cite{9})} If $|G|=p^km$, where $p\nmid m$, and $[G:{\bf O}^p(G)]=p^r$, then
\begin{equation}
|{\rm Max}_p(G)|\leq\frac{p^r-1}{p-1}+\frac{p^{k-r+1}-p}{p-1}\,.\nonumber
\end{equation}Moreover, $|{\rm Max}_p(G)|\leq\frac{p^{k+1}-p}{p-1}$\,, and if ${\bf O}^p(G)<G$, then
\begin{equation}
|{\rm Max}_p(G)|\leq\frac{p^k-1}{p-1}\,.\nonumber
\end{equation} 
\item[{\rm e)}]{\rm (B. Newton \cite{9})} If $|G|=p_1^{n_1}\cdots p_u^{n_u}$ is the decomposition of $|G|$ as a product of distinct primes and $p_1^{n_1}=\min\{p_i^{n_i}:1\leq i\leq u\}$, then
\begin{equation}
|{\rm Max}(G)|\leq\frac{p_1^{n_1}-1}{p_1-1}+\sum_{i=2}^u\frac{p_i^{n_i+1}-p_i}{p_i-1}.\nonumber
\end{equation}
\end{itemize}
\end{thmy}
\smallskip

The next corollaries follow immediately from the proof of Theorem 1.1 and Corollaries 1.2 and 1.5 of \cite{3}.

\begin{corollary} 
Let $G$ be a finite solvable group. Then the subgroup graph of $G$ contains a vertex of degree at least $3|G|/4$ if and only if $G$ is an elementary abelian $2$-group.
\end{corollary}

\begin{corollary} Let $G$ be a finite solvable group. Then the subgroup graph of $G$ contains a vertex of degree $|G|/2$ if and only if one of the following holds:
\begin{itemize}
\item[{\rm 1)}] $G=S_3\times D_8\times E$ with $\exp(E)\leq 2$.
\item[{\rm 2)}] $G$ is an elementary abelian $2$-group.
\item[{\rm 3)}] $G=C_2^{s-1}\times C_4$, where $s\geq 1$.
\item[{\rm 4)}] $G$ is a generalized extraspecial $2$-group.
\end{itemize}
\end{corollary}

Inspired by these results, we formulate the following natural open pro\-blems.

\bigskip\noindent{\bf Open problems.} 
\begin{itemize}
\item[1.] Given a constant $c\in(0,1)$, classify finite groups $G$ containing a subgroup $H$ with $d_{L(G)^*}(H)\geq c|G|$.
\item[2.] Classify \textit{arbitrary} finite groups $G$ whose subgroup graph contains a vertex of degree greater than $|G|/2-1$.
\end{itemize}

Most of our notation is standard and will usually not be repeated here. For basic notions and results on groups we refer the reader to \cite{7}.

\section{Proof of the main result}

We start by proving two auxiliary results. The first one gives an upper bound for the degree of a vertex $H$ in the subgroup graph of a finite solvable group $G$. It depends on the orders of $H$ and $G$.

\begin{lemma}
Let $G$ be a finite solvable group of order $n$ and $H$ be a subgroup of order $d$ of $G$. Then
\begin{equation}
d_{L(G)^*}(H)\leq d+\frac{n}{d}-2.
\end{equation}Moreover, we have equality if and only if $H$ is normal in $G$ and both $H$ and $G/H$ are elementary abelian $2$-groups.
\end{lemma}

\begin{proof}
If $k=|{\rm Max}(H)|$, then, by Theorem B, a), we have 
\begin{equation}
k\leq d-1.
\end{equation}
Let $N_1$, ..., $N_r$ be the atoms of the lattice interval between $H$ and $G$. Then $|N_i|\geq 2|H|$ and so $|N_i\setminus H|\geq |H|$, for all $i=1,...,r$. We get 
\begin{align*}
n-d&=|G\setminus H|\geq \left|\left(\bigcup_{i=1}^r N_i\right)\setminus H\right|=\left|\bigcup_{i=1}^r \left(N_i\setminus H\right)\right|=\sum_{i=1}^r |N_i\setminus H|\\
&\geq r|H|=rd,
\end{align*}thus
\begin{equation}
r\leq \frac{n}{d}-1.
\end{equation}It is now clear that inequalities (2) and (3) lead to
\begin{equation}
d_{L(G)^*}(H)=k+r\leq d+\frac{n}{d}-2,\nonumber
\end{equation}as desired.

The second part of the proof uses Theorem B, b). If $H$ is normal in $G$ and both $H$ and $G/H$ are elementary abelian $2$-groups, then $k=d-1$ and $r=\frac{n}{d}-1$, implying that $d_{L(G)^*}(H)=d+\frac{n}{d}-2$. Conversely, assume that $d_{L(G)^*}(H)=d+\frac{n}{d}-2$. Then 
\begin{equation}
k=d-1 \mbox{ and } r=\frac{n}{d}-1.
\end{equation}The first equality in (4) shows that $H$ is an elementary abelian $2$-group. From the second one, we infer that $|N_i|=2|H|$, for all $i=1,...,r$, and $G=\bigcup_{i=1}^r N_i$. It follows that $H\lhd N_i$, that is $N_i\subseteq N_G(H)$, for all $i=1,...,r$, and therefore $G\subseteq N_G(H)$. Then $H$ is normal in $G$. Now, the second equality in (4) means $$\delta(G/H)=|G/H|-1$$and so $G/H$ is also an elementary abelian $2$-group by Corollary 1.2 of \cite{3}.

This completes the proof. 
\end{proof}

Since $d+\frac{n}{d}-2\leq n-1$ for all divisors $d$ of $n$, we infer the following corollary.

\begin{corollary}
Let $G$ be a finite solvable group and $H$ be a subgroup of $G$. Then
\begin{equation}
d_{L(G)^*}(H)\leq |G|-1.\nonumber
\end{equation}Moreover, we have
\begin{equation}
|E(L(G)^*)|\leq\frac{1}{2}\,|V(L(G)^*)|(|G|-1),\nonumber
\end{equation}where $V(L(G)^*)$ and $E(L(G)^*)$ denote the sets of vertices and of edges of the graph $L(G)^*$, respectively.
\end{corollary}

The second lemma establishes an elementary inequality which will be used in the proof of Theorem 1.1.

\begin{lemma}
Let $p_1$, $p_2$, $p_3$ be distinct primes and $n_1,n_2,n_3\in\mathbb{N}^*$. Then
\begin{equation}
\sum_{i=1}^3\frac{p_i^{n_i+1}-p_i}{p_i-1}\leq\frac{1}{2}\prod_{i=1}^3p_i^{n_i}.\nonumber 
\end{equation}
\end{lemma}

\begin{proof}
Assume $p_1<p_2<p_3$ and let $x=p_3^{n_3}$. Then the inequality is equivalent to
\begin{equation}
\left(\frac{1}{2}\,p_1^{n_1}p_2^{n_2}-\frac{p_3}{p_3-1}\right)x\geq\frac{p_1^{n_1+1}-p_1}{p_1-1}+\frac{p_2^{n_2+1}-p_2}{p_2-1}-\frac{p_3}{p_3-1}\,.\nonumber
\end{equation}Since
\begin{equation}
1<\frac{p_3}{p_3-1}\leq\frac{5}{4}\,,\nonumber
\end{equation}it suffices to prove that
\begin{equation}
\left(\frac{1}{2}\,p_1^{n_1}p_2^{n_2}-\frac{5}{4}\right)x\geq\frac{p_1^{n_1+1}-p_1}{p_1-1}+\frac{p_2^{n_2+1}-p_2}{p_2-1}-1\nonumber
\end{equation}and since $x\geq p_3$, it suffices to prove that
\begin{equation}
\left(\frac{1}{2}\,p_1^{n_1}p_2^{n_2}-\frac{5}{4}\right)p_3\geq\frac{p_1^{n_1+1}-p_1}{p_1-1}+\frac{p_2^{n_2+1}-p_2}{p_2-1}-1.\nonumber
\end{equation}Let $y=p_2^{n_2}$. Then the above inequality is equivalent to
\begin{equation}
\left(\frac{1}{2}\,p_1^{n_1}p_3-\frac{p_2}{p_2-1}\right)y\geq\frac{p_1^{n_1+1}-p_1}{p_1-1}+\frac{5}{4}\,p_3-\frac{p_2}{p_2-1}-1.\nonumber
\end{equation}Similarly, since 
\begin{equation}
1<\frac{p_2}{p_2-1}\leq\frac{3}{2}\nonumber
\end{equation}and $y\geq p_2$, it suffices to show that
\begin{equation}
\left(\frac{1}{2}\,p_1^{n_1}p_3-\frac{3}{2}\right)p_2\geq\frac{p_1^{n_1+1}-p_1}{p_1-1}+\frac{5}{4}\,p_3-2.\nonumber
\end{equation}Let $z=p_1^{n_1}$. Then the above inequality is equivalent to
\begin{equation}
\left(\frac{1}{2}\,p_2p_3-\frac{p_1}{p_1-1}\right)z\geq\frac{3}{2}\,p_2+\frac{5}{4}\,p_3-\frac{p_1}{p_1-1}-2.\nonumber
\end{equation}Similarly, since 
\begin{equation}
1<\frac{p_1}{p_1-1}\leq 2\nonumber
\end{equation}and $z\geq p_1\geq 2$, it suffices to show that
\begin{equation}
\left(\frac{1}{2}\,p_2p_3-2\right)\cdot 2\geq\frac{3}{2}\,p_2+\frac{5}{4}\,p_3-3,\nonumber
\end{equation}or equivalently
\begin{equation}
p_2p_3\geq\frac{3}{2}\,p_2+\frac{5}{4}\,p_3+1.\nonumber
\end{equation}This becomes
\begin{equation}
p_2\left(p_3-\frac{3}{2}\right)\geq\frac{5}{4}\,p_3+1.\nonumber
\end{equation}Finally, since $p_2\geq 3$, it suffices to show that
\begin{equation}
3\cdot\left(p_3-\frac{3}{2}\right)\geq\frac{5}{4}\,p_3+1,\nonumber
\end{equation}that is $p_3\geq\frac{22}{7}$\,, which is true, completing the proof.
\end{proof}

We are now able to prove our main result.

\bigskip\noindent{\bf Proof of Theorem 1.1.} Recall that $G$ is a finite solvable group with $|G|=n$ and there exists a subgroup $H$ of $G$ with $|H|=d$, such that $\frac{n}{2}-1<d_{L(G)^*}(H)$. Then by Lemma 2.1 we have
\begin{equation}
\frac{n}{2}-1<d+\frac{n}{d}-2,\nonumber
\end{equation}that is
\begin{equation}
2d^2-(n+2)d+2n>0.
\end{equation}Then either the discriminant $n^2-12n+4$ of (5) is negative, i.e. $n\leq 11$, or $d\in[1,d_1)\cup (d_2,\infty)$, where
\begin{equation}
d_{1,2}=\frac{n+2\mp\sqrt{n^2-12n+4}}{4}\,.\nonumber
\end{equation}In the first case, it is easy to see that all groups of order at most $11$ except $C_n$ for $n=5,...,11$ are solutions of our problem, that is each of these groups $G$ admits a subgroup $H$ which satisfies $\frac{n}{2}-1<d_{L(G)^*}(H)$. So, in what follows we will assume that $n\geq 12$. Then, by analytical methods, we get $d_1\leq 3$ and $d_2\geq\frac{n}{3}$, implying that $d\in\{1,2,\frac{n}{2}\,,n\}$.
\bigskip

\hspace{10mm}{\bf Case 1.} $d=1$
\smallskip

Let $r$ be the number of atoms of the lattice interval between $H$ and $G$. So, in this particular case, $r$ is the number of atoms in the
subgroup lattice of $G$. Then $r=d_{L(G)^*}(H)>\frac{n}{2}-1$ and we get the solvable groups in Theorem A.
\bigskip

\hspace{10mm}{\bf Case 2.} $d=2$
\smallskip

Then $d_{L(G)^*}(H)\geq\frac{n}{2}$\,. On the other hand, by Lemma 2.1 we have 
$d_{L(G)^*}(H)\leq\frac{n}{2}$\,. Thus, in this case, $d_{L(G)^*}(H)=\frac{n}{2}$ and $H$ is normal in $G$ and $G/H$ is an elementary abelian $2$-group, say $G/H\cong C_2^s$. Moreover, $G$ is a $2$-group and $G'\subseteq H$.

If $G'=1$ we get $G\cong C_2^{s+1}$ or $G\cong C_2^{s-1}\times C_4$, where $s\geq 3$, while if $G'=H$ we get $G'=\Phi(G)\subseteq Z(G)$, i.e. $G$ is a generalized extraspecial $2$-group.
\bigskip

\hspace{10mm}{\bf Case 3.} $d=\frac{n}{2}$
\smallskip

Let $k=|{\rm Max}(H)|$. Then the inequality
\begin{equation}
1+k=d_{L(G)^*}(H)>\frac{n}{2}-1\nonumber 
\end{equation}means $k\geq\frac{n}{2}-1$ and so $H\cong C_2^t$ for some positive integer $t$ by Theorem B, b). It follows that $n=2^{t+1}$ and $i_2(G)\geq i_2(H)=2^t$, i.e. $i_2(G)>\frac{|G|}{2}-1$, where $i_2(G)$ denotes the number of involutions of $G$. Thus we get the groups classified by Wall \cite{14}, i.e. the groups of types (I)-(IV) in Theorem A.
\bigskip

\hspace{10mm}{\bf Case 4.} $d=n$
\smallskip

Then 
\begin{equation}
k=d_{L(G)^*}(H)>\frac{n}{2}-1.
\end{equation}Let $n=p_1^{n_1}\cdots p_u^{n_u}$ be the decomposition of $n$ as a product of distinct primes and assume that $p_1^{n_1}=\min\{p_i^{n_i}:1\leq i\leq u\}$. Using Theorem B, e), we obtain
\begin{equation}
k\leq\frac{p_1^{n_1}-1}{p_1-1}+\sum_{i=2}^u\frac{p_i^{n_i+1}-p_i}{p_i-1}\leq\sum_{i=1}^u\frac{p_i^{n_i+1}-p_i}{p_i-1}-1.\nonumber 
\end{equation}Now we prove by induction that for all $u\geq 3$ we have
\begin{equation}
\sum_{i=1}^u\frac{p_i^{n_i+1}-p_i}{p_i-1}\leq\frac{n}{2}\nonumber 
\end{equation}and so $k\leq\frac{n}{2}-1$, contradicting (6). For $u=3$, the inequality holds from Lemma 2.3. Assume that it holds for an arbitrary $u\geq 3$. By inductive hypothesis, we get
\begin{equation}
\sum_{i=1}^{u+1}\frac{p_i^{n_i+1}-p_i}{p_i-1}\leq\frac{1}{2}\,p_1^{n_1}\cdots p_u^{n_u}+\frac{p_{u+1}^{n_{u+1}+1}-p_{u+1}}{p_{u+1}-1}\,.
\end{equation}The fact that the right side of (7) is at most $\frac{1}{2}\,p_1^{n_1}\cdots p_u^{n_u}p_{u+1}^{n_{u+1}}$ is equivalent with
\begin{equation}
\frac{p_{u+1}^{n_{u+1}+1}-p_{u+1}}{p_{u+1}-1}\leq\frac{1}{2}\,p_1^{n_1}\cdots p_u^{n_u}\left(p_{u+1}^{n_{u+1}}-1\right),\nonumber
\end{equation}i.e. with
\begin{equation}
\frac{p_{u+1}}{p_{u+1}-1}\leq\frac{1}{2}\,p_1^{n_1}\cdots p_u^{n_u},\nonumber
\end{equation}which is obviously true.

Thus $u\leq 2$. Also, we have $2\mid n$ and $\Phi(G)=1$ by Theorem B, c), and the condition $k\geq n/2$.

If $u=1$, then $G$ is obviously an elementary abelian $2$-group. Assume that $u=2$. Then $n=p_1^{n_1}p_2^{n_2}$ with $p_1=2$ or $p_2=2$. For $p_1=2$, the inequality $n/2\leq k$ becomes
\begin{equation}
2^{n_1-1}p_2^{n_2}\leq 2^{n_1}-1+\frac{p_2^{n_2+1}-p_2}{p_2-1}\,.\nonumber
\end{equation}Then $n_1=1$, i.e. $n=2\cdot p_2^{n_2}$. Let $P_2$ be a Sylow $p_2$-subgroup of $G$. Since $P_2$ is normal in $G$, we obtain $\Phi(P_2)=1$, that is $P_2\cong C_{p_2}^{n_2}$. This shows that $G\cong C_{p_2}^{n_2}\rtimes C_2$. For $p_2=2$, the inequality $n/2\leq k$ becomes
\begin{equation}
p_1^{n_1}2^{n_2-1}\leq\frac{p_1^{n_1}-1}{p_1-1}+2^{n_2+1}-2,\nonumber
\end{equation}implying that $p_1=3$ and $n_1=1$, i.e. $n=3\cdot 2^{n_2}$ with $n_2\geq 2$. Clearly, $G$ is not $A_4$, and since
$\Phi(G)=1$, the group $G$ possesses a normal subgroup of index $2$. It follows that ${\bf O}^2(G)<G$ and so
\begin{equation}
|{\rm Max}_2(G)|\leq 2^{n_2}-1\nonumber
\end{equation}by Theorem B, d). Since $G$ is solvable, its maximal subgroups are of prime power index, which leads to 
\begin{equation}
k=|{\rm Max}_3(G)|+|{\rm Max}_2(G)|.\nonumber
\end{equation}If $G$ has a unique Sylow $2$-subgroup, i.e. $|{\rm Max}_3(G)|=1$, then we get
\begin{equation}
3\cdot 2^{n_2-1}\leq k=1+|{\rm Max}_2(G)|\leq 2^{n_2},\nonumber
\end{equation}a contradiction. If $G$ has three Sylow $2$-subgroups, i.e. $|{\rm Max}_3(G)|=3$, then we get
\begin{equation}
3\cdot 2^{n_2-1}\leq k=3+|{\rm Max}_2(G)|\leq 2^{n_2}+2,\nonumber
\end{equation}that is $n_2=2$. Thus $n=12$ and the unique solution of our problem is $G\cong D_{12}$.

The proof of Theorem 1.1 is now complete.$\qed$
\bigskip

\noindent{\bf Acknowledgments.} The author is grateful to the reviewer for remarks which improve the previous version of the paper.

\vspace*{3ex}\small

\hfill
\begin{minipage}[t]{5cm}
Marius T\u arn\u auceanu \\
Faculty of  Mathematics \\
``Al.I. Cuza'' University \\
Ia\c si, Romania \\
e-mail: {\tt tarnauc@uaic.ro}
\end{minipage}

\end{document}